\documentclass[paper=a4,english,fontsize=11pt,parskip=half,abstract=true]{scrartcl}
\usepackage{babel}
\usepackage[utf8]{inputenc}
\usepackage[T1]{fontenc}
\usepackage[left=20mm,right=20mm,top=30mm,bottom=30mm]{geometry}
\usepackage{amsmath}
\usepackage{amsthm}
\usepackage{amssymb}
\usepackage{enumerate} 
\usepackage{thmtools}
\usepackage{mathtools}
\mathtoolsset{centercolon} 
\usepackage[bookmarks=true,
            pdftitle={Real characters in nilpotent blocks},
            pdfauthor={Benjamin Sambale},
            pdfkeywords={},
            pdfstartview={FitH}]{hyperref}

\newtheorem*{ThmA}{Theorem A}
\newtheorem*{ConB}{Conjecture B}
\newtheorem*{ConC}{Conjecture C}
\newtheorem*{ThmD}{Theorem D}
\newtheorem*{ThmE}{Theorem E}
\newtheorem{Thm}{Theorem} 
\newtheorem{Lem}[Thm]{Lemma}
\newtheorem{Prop}[Thm]{Proposition}
\newtheorem{Cor}[Thm]{Corollary}
\newtheorem{Con}[Thm]{Conjecture}
\theoremstyle{definition}
\newtheorem{Def}[Thm]{Definition}
\newtheorem{Ex}[Thm]{Example}

\allowdisplaybreaks[1]

\renewcommand{\phi}{\varphi}
\newcommand{\C}{\mathrm{C}}
\newcommand{\N}{\mathrm{N}}
\newcommand{\Z}{\mathrm{Z}}
\newcommand{\pcore}{\mathrm{O}}

\newcommand{\ZZ}{\mathbb{Z}}

\newcommand{\RR}{\mathbb{R}}

\newcommand{\Aut}{\mathrm{Aut}}

\newcommand{\Irr}{\mathrm{Irr}}
\newcommand{\IBr}{\mathrm{IBr}}

\newcommand{\Ker}{\operatorname{Ker}}
\newcommand{\Cl}{\operatorname{Cl}}

\newcommand{\sgn}{\operatorname{sgn}}
\newcommand{\diag}{\operatorname{diag}}

\newcommand{\TT}{\mathrm{t}}

\title{Real characters in nilpotent blocks}
\author{Benjamin Sambale\footnote{Institut für Algebra, Zahlentheorie und Diskrete Mathematik, Leibniz Universität Hannover, Welfengarten 1, 30167 Hannover, Germany,
\href{mailto:sambale@math.uni-hannover.de}{sambale@math.uni-hannover.de}}}
\date{\today}
\dedication{Dedicated to Pham Huu Tiep on the occasion of his 60th birthday.}

\begin{document}
\frenchspacing
\maketitle
\begin{abstract}\noindent
We prove that the number of irreducible real characters in a nilpotent block of a finite group is locally determined. We further conjecture that the Frobenius--Schur indicators of those characters can be computed for $p=2$ in terms of the extended defect group. 
We derive this from a more general conjecture on the Frobenius--Schur indicator of projective indecomposable characters of $2$-blocks with one simple module. This extends results of Murray on $2$-blocks with cyclic and dihedral defect groups.
\end{abstract}


\renewcommand{\sectionautorefname}{Section}
\section{Introduction}

An important task in representation theory is to determine global invariants of a finite group $G$ by means of local subgroups. 
Dade's conjecture, for instance, predicts the number of irreducible characters $\chi\in\Irr(G)$ such that the $p$-part $\chi(1)_p$ is a given power of a prime $p$  (see \cite[Conjecture~9.25]{Navarro2}). Since Gow's work~\cite{GowFirst}, there has been an increasing interest in counting real (i.\,e. real-valued) characters and more generally characters with a given field of values. 

The quaternion group $Q_8$ testifies that a real irreducible character $\chi$ is not always afforded by a representation over the real numbers. The precise behavior is encoded by the \emph{Frobenius--Schur indicator} (F-S indicator, for short)
\begin{equation}\label{FS}
\epsilon(\chi):=\frac{1}{|G|}\sum_{g\in G}\chi(g^2)=\begin{cases}
0&\text{if }\overline{\chi}\ne\chi,\\
1&\text{if $\chi$ is realized by a real representation},\\
-1&\text{if $\chi$ is real, but not realized by a real representation}.
\end{cases}
\end{equation}
A new interpretation of the F-S indicator in terms of superalgebras has been given recently in \cite{Ichikawa}.
The case of the dihedral group $D_8$ shows that $\epsilon(\chi)$ is not determined by the character table of $G$. The computation of F-S indicators can be a surprisingly difficult task, which has not been fully completed for the simple groups of Lie type, for instance (see \cite{TrefethenVinroot}). Problem~14 on Brauer's famous list~\cite{BrauerLectures} asks for a group-theoretical interpretation of the number of $\chi\in\Irr(G)$ with $\epsilon(\chi)=1$. 

To obtain deeper insights, we fix a prime $p$ and assume that $\chi$ lies in a $p$-block $B$ of $G$ with defect group $D$. By complex conjugation we obtain another block $\overline{B}$ of $G$. If $\overline{B}\ne B$, then clearly $\epsilon(\chi)=0$ for all $\chi\in\Irr(B)$. Hence, we assume that $B$ is real, i.\,e. $\overline{B}=B$. 
John Murray~\cite{MurrayCyclic2,MurraySubpairs} has computed the F-S indicators when $D$ is a cyclic $2$-group or a dihedral $2$-group (including the Klein four-group). His results depend on the fusion system of $B$, on Erdmann's classification of tame blocks and on the structure of the so-called \emph{extended defect group} $E$ of $B$ (see \autoref{defdefectpair} below). 
For $p>2$ and $D$ cyclic, he obtained in \cite{MurrayCyclicOdd} partial information on the F-S indicators in terms of the Brauer tree of $B$.

The starting point of my investigation is the well-known fact that $2$-blocks with cyclic defect groups are nilpotent. Assume that $B$ is nilpotent and real. If $B$ is the principal block, then $G=\pcore_{p'}(G)D$ and $\Irr(B)=\Irr(G/\pcore_{p'}(G))=\Irr(D)$. In this case the F-S indicators of $B$ are determined by $D$ alone. Thus, suppose that $B$ is non-principal.
By Broué--Puig~\cite{BrouePuig}, there exists a height-preserving bijection $\Irr(D)\to\Irr(B)$, $\lambda\mapsto\lambda*\chi_0$ where $\chi_0\in\Irr(B)$ is a fixed character of height $0$ (see also \cite[Definition 8.10.2]{LinckelmannBook2}). However, this bijection does not in general preserve F-S indicators. For instance, the dihedral group $D_{24}$ has a nilpotent $2$-block with defect group $C_4$ and a nilpotent $3$-block with defect group $C_3$, although every character of $D_{24}$ is real. 
Our main theorem asserts that the number of real characters in a nilpotent block is nevertheless locally determined.
To state it, we introduce the \emph{extended inertial group} 
\[\N_G(D,b_D)^*:=\bigl\{g\in\N_G(D):b_D^g\in\{b_D,\overline{b_D}\}\bigr\}\]
where $b_D$ is a Brauer correspondent of $B$ in $D\C_G(D)$.  

\begin{ThmA}
Let $B$ be a real, nilpotent $p$-block of a finite group $G$ with defect group $D$. Let $b_D$ be a Brauer correspondent of $B$ in $D\C_G(D)$. Then the number of real characters in $\Irr(B)$ of height $h$ coincides with the number of characters $\lambda\in\Irr(D)$ of degree $p^h$ such that $\lambda^t=\overline{\lambda}$ where 
\[\N_G(D,b_D)^*/D\C_G(D)=\langle tD\C_G(D)\rangle.\]
If $p>2$, then all real characters in $\Irr(B)$ have the same F-S indicator.
\end{ThmA}

In contrast to arbitrary blocks, Theorem~A implies that nilpotent real blocks have at least one real character (cf. \cite[p. 92]{MurrayCyclicOdd} and \cite[Theorem~5.3]{Gow}).
If $\overline{b_D}=b_D$, then $B$ and $D$ have the same number of real characters, because $\N_G(D,b_D)=D\C_G(D)$. This recovers a result of Murray~\cite[Lemma~2.2]{MurrayCyclic2}. As another consequence, we will derive in \autoref{HHS} a real version of Eaton's conjecture~\cite{Eaton} for nilpotent blocks as put forward by Héthelyi--Horváth--Szabó~\cite{HHS}.

The F-S indicators of real characters in nilpotent blocks seem to lie somewhat deeper. We still conjecture that they are locally determined by a defect pair (see \autoref{defdefectpair}) for $p=2$ as follows.

\begin{ConB}
Let $B$ be a real, nilpotent, non-principal $2$-block of a finite group $G$ with defect pair $(D,E)$. 
Then there exists a height preserving bijection $\Gamma:\Irr(D)\to\Irr(B)$ such that
\begin{equation}\label{GowInd}
\epsilon(\Gamma(\lambda))=\frac{1}{|D|}\sum_{e\in E\setminus D}\lambda(e^2)
\end{equation}
for all $\lambda\in\Irr(D)$.
\end{ConB}

The right hand side of \eqref{GowInd} was introduced and studied by Gow~\cite[Lemma~2.1]{Gow} more generally for any groups $D\le E$ with $|E:D|=2$. This invariant was later coined the \emph{Gow indicator} by Murray~\cite[Eq. (2)]{MurrayCyclicOdd}. 
For $2$-blocks of defect $0$, Conjecture~B confirms the known fact that real characters of $2$-defect $0$ have F-S indicator $1$ (see \cite[Theorem~5.1]{Gow}). There is no such result for odd primes $p$. As a matter of fact, every real character has $p$-defect $0$ whenever $p$ does not divide $|G|$. In \autoref{abel} we prove Conjecture~B for abelian defect groups $D$. Then it also holds for all quasisimple groups $G$ by work of An--Eaton~\cite{AnEaton2}.
Murray's results mentioned above, imply Conjecture~B also for dihedral $D$. 

For $p>2$, the common F-S indicator in the situation of Theorem~A is not locally determined. For instance, $G=Q_8\rtimes C_9=\mathtt{SmallGroup}(72,3)$ has a non-principal real $3$-block with $D\cong C_9$ and common F-S indicator $-1$, while its Brauer correspondent in $\N_G(D)\cong C_{18}$ has common F-S indicator $1$. Nevertheless, for cyclic defect groups $D$ we find another way to compute this F-S indicator in \autoref{podd} below. 

Our second conjecture applies more generally to blocks with only one simple module. 

\begin{ConC}
Let $B$ be a real, non-principal $2$-block with defect pair $(D,E)$ and a unique projective indecomposable character $\Phi$. 
Then 
\[\epsilon(\Phi)=|\{x\in E\setminus D:x^2=1\}|.\]
\end{ConC}

Here $\epsilon(\Phi)$ is defined by extending \eqref{FS} linearly.
If $\epsilon(\Phi)=0$, then $E$ does not split over $D$ and Conjecture~C holds (see \autoref{GowMurray} below).
Conjecture~C implies a stronger, but more technical statement on $2$-blocks with a Brauer correspondent with one simple module (see \autoref{loc} below). This allows us to prove the following.

\begin{ThmD}
Conjecture~C implies Conjecture~B.
\end{ThmD}

We remark that our proof of Theorem~D does not work block-by-block.
For solvable groups we offer a purely group-theoretical version of Conjecture~C at the end of \autoref{secB}.

\begin{ThmE}
Conjectures B and C hold for all nilpotent $2$-blocks of solvable groups. 
\end{ThmE}

We have checked Conjectures~B and C with GAP~\cite{GAPnew} in many examples using the libraries of small groups, perfect groups and primitive groups. 

\section{Theorem~A and its consequences}

Our notation follows closely Navarro's book~\cite{Navarro}.
Let $B$ be a $p$-block of a finite group $G$ with defect group $D$. 
Recall that a $B$-\emph{subsection} is a pair $(u,b)$ where $u\in D$ and $b$ is a Brauer correspondent of $B$ in $\C_G(u)$. For $\chi\in\Irr(B)$ and $\phi\in\IBr(b)$ we denote the corresponding generalized decomposition number by $d^u_{\chi\phi}$. If $u=1$, we obtain the (ordinary) decomposition number $d_{\chi\phi}=d^1_{\chi\phi}$. We put $l(b)=|\IBr(b)|$ as usual.

Following \cite[p. 114]{Navarro}, we define a class function $\chi^{(u,b)}$ by 
\[\chi^{(u,b)}(us):=\sum_{\phi\in\IBr(b)}d^u_{\chi\phi}\phi(s)\] 
for $s\in\C_G(u)^0$ and $\chi^{(u,b)}(x)=0$ whenever $x$ is outside the $p$-section of $u$. If $\mathcal{R}$ is a set of representatives for the $G$-conjugacy classes of $B$-subsections, then $\chi=\sum_{(u,b)\in\mathcal{R}}\chi^{(u,b)}$ by Brauer's second main theorem (see \cite[Problem~5.3]{Navarro}). Now suppose that $B$ is nilpotent and $\lambda\in\Irr(D)$. By \cite[Proposition~8.11.4]{LinckelmannBook2}, each Brauer correspondent $b$ of $B$ is nilpotent and in particular $l(b)=1$. 
Broué--Puig~\cite{BrouePuig} have shown that, if $\chi$ has height $0$, then
\[\lambda*\chi:=\sum_{(u,b)\in\mathcal{R}}\lambda(u)\chi^{(u,b)}\in\Irr(B)\]
and $(\lambda*\chi)(1)=\lambda(1)\chi(1)$. Note also that $d_{\lambda*\chi,\phi}^u=\lambda(u)d^u_{\chi\phi}$. 

\begin{proof}[Proof of Theorem~A]
Let $\mathcal{R}$ be a set of representatives for the $G$-conjugacy classes of $B$-subsections $(u,b_u)\le (D,b_B)$ (see \cite[p. 219]{Navarro}). Since $B$ is nilpotent, we have $\IBr(b_u)=\{\phi_u\}$ for all $(u,b_u)\in\mathcal{R}$.
Since the Brauer correspondence is compatible with complex conjugation, $(u,\overline{b_u})^t\le(D,\overline{b_D})^t=(D,b_D)$ where $\N_G(D,b_D)^*/D\C_G(D)=\langle tD\C_G(D)\rangle$. Thus, $(u,\overline{b_u})^t$ is $D$-conjugate to some $(u',b_{u'})\in\mathcal{R}$.

If $p>2$, there exists a unique $p$-rational character $\chi_0\in\Irr(B)$ of height $0$, which must be real by uniqueness (see \cite[Remark after Theorem~1.2]{BrouePuig}). If $p=2$, there is a $2$-rational real character $\chi_0\in\Irr(B)$ of height $0$ by \cite[Theorem~5.1]{Gow}. 
Then $d_{\chi_0,\phi_u}^u=d_{\chi_0,\overline{\phi_u}}^u\in\ZZ$ and
\[\overline{\chi_0^{(u,b_u)}}=\chi_0^{(u,\overline{b_u})}=\chi_0^{(u,\overline{b_u})^t}=\chi_0^{(u',b_{u'})}.\]

Now let $\lambda\in\Irr(D)$. Then 
\[\overline{\lambda*\chi_0}=\sum_{(u,b_u)\in\mathcal{R}}\overline{\lambda}(u)\overline{\chi_0^{(u,b_u)}}=
\sum_{(u,b_u)\in\mathcal{R}}\overline{\lambda}(u)\chi_0^{(u',b_{u'})}.\]
Since the class functions $\chi_0^{(u,b)}$ have disjoint support, they are linearly independent. Therefore, $\lambda*\chi_0$ is real if and only if $\lambda(u^t)=\lambda(u')=\overline{\lambda}(u)$ for all $(u,b_u)\in\mathcal{R}$. Since every conjugacy class of $D$ is represented by some $u$ with $(u,b_u)\in\mathcal{R}$, we conclude that $\lambda*\chi_0$ is real if and only $\lambda^t=\overline{\lambda}$. Moreover, if $\lambda(1)=p^h$, then $\lambda*\chi_0$ has height $h$. This proves the first claim.

To prove the second claim, let $p>2$ and $\IBr(B)=\{\phi\}$. Then the decomposition numbers $d_{\lambda*\chi_0,\phi}=\lambda(1)$ are powers of $p$; in particular they are odd. A theorem of Thompson and Willems (see \cite[Theorem~2.8]{WillemsDual}) states that all real characters $\chi$ with $d_{\chi,\phi}$ odd have the same F-S indicator. So in our situation all real characters in $\Irr(B)$ have the same F-S indicator. 
\end{proof}

Since the automorphism group of a $p$-group is “almost always” a $p$-group (see \cite{Helleloid}), the following consequence is of interest.

\begin{Cor}
Let $B$ be a real, nilpotent $p$-block with defect group $D$ such that $p$ and $|\Aut(D)|$ are odd. Then $B$ has a unique real character. 
\end{Cor}
\begin{proof}
The hypothesis on $\Aut(D)$ implies that $\N_G(D,b_D)^*=D\C_G(D)$. Hence by Theorem~A, the number of real characters in $\Irr(B)$ is the number of real characters in $D$. Since $p>2$, the trivial character is the only real character of $D$. 
\end{proof}

The next lemma is a consequence of Brauer's second main theorem and the fact that $|\{g\in G:g^2=x\}|=|\{g\in\C_G(x):g^2=x\}|$ is locally determined for $g,x\in G$. 

\begin{Lem}[Brauer]\label{lembrauer}
For every $p$-block $B$ of $G$ and every $B$-subsection $(u,b)$ with $\phi\in\IBr(b)$ we have
\[\sum_{\chi\in\Irr(B)}\epsilon(\chi)d^u_{\chi\phi}=\sum_{\psi\in\Irr(b)}\epsilon(\psi)d^u_{\psi\phi}=\sum_{\psi\in\Irr(b)}\epsilon(\psi)\frac{\psi(u)}{\psi(1)}d_{\psi\phi}.\]
If $l(b)=1$, then 
\[\sum_{\chi\in\Irr(B)}\epsilon(\chi)d^u_{\chi\phi}=\frac{1}{\phi(1)}\sum_{\psi\in\Irr(b)}\epsilon(\psi)\psi(u).\]
\end{Lem}
\begin{proof}
The first equality is \cite[Theorem~4A]{BrauerApp3}. The second follows from $u\in\Z(\C_G(u))$. If $l(b)=1$, then $\psi(1)=d_{\psi\phi}\phi(1)$ for $\psi\in\Irr(b)$ and the last claim follows.
\end{proof}

Recall that a \emph{canonical} character of $B$ is a character $\theta\in\Irr(D\C_G(D))$ lying in a Brauer correspondent of $B$ such that $D\le\Ker(\theta)$ (see \cite[Theorem~9.12]{Navarro}). We define the \emph{extended stabilizer} 
\[\N_G(D)_\theta^*:=\bigl\{g\in\N_G(D):\theta^g\in\{\theta,\overline{\theta}\}\bigr\}.\]
The following results adds some detail to the nilpotent case of \cite[Theorem~1]{MurrayCyclicOdd}.

\begin{Thm}\label{podd}
Let $B$ be a real, nilpotent $p$-block with cyclic defect group $D=\langle u\rangle$ and $p>2$. Let $\theta\in\Irr(\C_G(D))$ be a canonical character of $B$ and set $T:=\N_G(D)_\theta^*$. Then one of the following holds:
\begin{enumerate}[(1)]
\item $\overline{\theta}\ne\theta$. All characters in $\Irr(B)$ are real with F-S indicator $\epsilon(\theta^T)$. 

\item $\overline{\theta}=\theta$. The unique non-exceptional character $\chi_0\in\Irr(B)$ is the only real character in $\Irr(B)$ and $\epsilon(\chi_0)=\sgn(\chi_0(u))\epsilon(\theta)$ where $\sgn(\chi_0(u))$ is the sign of $\chi_0(u)$. 
\end{enumerate}
\end{Thm}
\begin{proof}
Let $b_D$ be a Brauer correspondent of $B$ in $\C_G(D)$ containing $\theta$. Then $T=\N_G(D,b_D)^*$. If $\overline{\theta}\ne\theta$, then $T$ inverts the elements of $D$ since $p>2$. Thus, Theorem~A implies that all characters in $\Irr(B)$ are real. By \cite[Theorem~1(v)]{MurrayCyclicOdd}, the common F-S indicator is the Gow indicator of $\theta$ with respect to $T$. This is easily seen to be $\epsilon(\theta^T)$ (see \cite[after Eq. (2)]{MurrayCyclicOdd}).

Now assume that $\overline{\theta}=\theta$. Here Theorem~A implies that the unique $p$-rational character $\chi_0\in\Irr(B)$ is the only real character. In particular, $\chi_0$ must be the unique non-exceptional character. 
Note that $(u,b_D)$ is a $B$-subsection and $\IBr(b_D)=\{\phi\}$. Since $\chi_0$ is $p$-rational, $d^u_{\chi_0\phi}=\pm1$. Since all Brauer correspondents of $B$ in $\C_G(u)$ are conjugate under $\N_G(D)$, the generalized decomposition numbers are Galois conjugate, in particular $d^u_{\chi_0\phi}$ does not depend on the choice of $b_D$. 
Hence, 
\[\chi_0(u)=|\N_G(D):\N_G(D)_\theta|d^u_{\chi_0\phi}\phi(1)\] 
and $d^u_{\chi_0\phi}=\sgn(\chi_0(u))$.
Moreover, $\theta$ is the unique non-exceptional character of $b_D$ and $\theta(u)=\theta(1)$. By \autoref{lembrauer}, we obtain 
\[\epsilon(\chi_0)=\sgn(\chi_0(u))\sum_{\chi\in\Irr(B)}\epsilon(\chi)d^u_{\chi\phi}=\frac{\sgn(\chi_0(u))}{\phi(1)}\sum_{\psi\in\Irr(b_D)}\epsilon(\psi)\psi(u)=\sgn(\chi_0(u))\epsilon(\theta).\qedhere\] 
\end{proof}

If $B$ is a nilpotent block with canonical character $\theta\ne\overline{\theta}$, the common F-S indicator of the real characters in $\Irr(B)$ is not always $\epsilon(\theta^T)$ as in \autoref{podd}. A counterexample is given by a certain $3$-block of $G=\mathtt{SmallGroup}(288,924)$ with defect group $D\cong C_3\times C_3$. 

We now restrict ourselves to $2$-blocks. 
Héthelyi--Horváth--Szabó~\cite{HHS} introduced four conjectures, which are real versions of Brauer's conjecture, Olsson's conjecture and Eaton's conjecture. We only state the strongest of them, which implies the remaining three. Let $D^{(0)}:=D$ and $D^{(k+1)}:=[D^{(k)},D^{(k)}]$ for $k\ge 0$ be the members of the derived series of $D$. 

\begin{Con}[Héthelyi--Horváth--Szabó]\label{conHHS}
Let $B$ be a $2$-block with defect group $D$. For every $h\ge 0$, the number of real characters in $\Irr(B)$ of height $\le h$ is bounded by the number of elements of $D/D^{(h+1)}$ which are real in $\N_G(D)/D^{(h+1)}$.
\end{Con}

A conjugacy class $K$ of $G$ is called \emph{real} if $K=K^{-1}:=\{x^{-1}:x\in K\}$.
A conjugacy class $K$ of a normal subgroup $N\unlhd G$ is called \emph{real under} $G$ if there exists $g\in G$ such that $K^g=K^{-1}$. 

\begin{Prop}\label{HHS}
Let $B$ be a nilpotent $2$-block with defect group $D$ and Brauer correspondent $b_D$ in $D\C_G(D)$. Then the number of real characters in $\Irr(B)$ of height $\le h$ is bounded by the number of conjugacy classes of $D/D^{(h+1)}$ which are real under $\N_G(D,b_D)^*/D^{(h+1)}$.
In particular, \autoref{conHHS} holds for $B$.  
\end{Prop}
\begin{proof}
We may assume that $B$ is real. As in the proof of Theorem~A, we fix some $2$-rational real character $\chi_0\in\Irr(B)$ of height $0$. Now $\lambda*\chi_0$ has height $\le h$ if and only if $\lambda(1)\le p^h$ for $\lambda\in\Irr(B)$. By \cite[Theorem~5.12]{Isaacs}, the characters of degree $\le p^h$ in $\Irr(D)$ lie in $\Irr(D/D^{(h+1)})$. By Theorem~A, $\lambda*\chi_0$ is real if and only if $\lambda^t=\overline{\lambda}$. By Brauer's permutation lemma (see \cite[Theorem~2.3]{Navarro2}), the number of those characters $\lambda$ coincides with the number of conjugacy classes $K$ of $D/D^{(h+1)}$ such that $K^t=K^{-1}$. 
Now \autoref{conHHS} follows from $\N_G(D,b_D)^*\le\N_G(D)$.
\end{proof}

\section{Extended defect groups}\label{secext}

We continue to assume that $p=2$. As usual we choose a complete discrete valuation ring $\mathcal{O}$ such that $F:=\mathcal{O}/J(\mathcal{O})$ is an algebraically closed field of characteristic $2$. Let $\Cl(G)$ be the set of conjugacy classes of $G$. For $K\in\Cl(G)$ let $K^+:=\sum_{x\in K}x\in\Z(FG)$ be the class sum of $K$. 
We fix a $2$-block $B$ of $FG$ with block idempotent $1_B=\sum_{K\in\Cl(G)}a_KK^+$ where $a_K\in F$. The central character of $B$ is defined by 
\[\lambda_B:\Z(FG)\to F,\quad K^+\mapsto\Bigl(\frac{|K|\chi(g)}{\chi(1)}\Bigr)^*\] 
where $g\in K$, $\chi\in\Irr(B)$ and $^*$ denotes the canonical reduction $\mathcal{O}\to F$ (see \cite[Chapter~2]{Navarro}).

Since $\lambda_B(1_B)=1$, there exists $K\in\Cl(G)$ such that $a_K\ne 0\ne\lambda_B(K^+)$. We call $K$ a \emph{defect class} of $B$. By \cite[Corollary~3.8]{Navarro}, $K$ consists of elements of odd order. According to \cite[Corollary~4.5]{Navarro}, a Sylow $2$-subgroup $D$ of $\C_G(x)$ where $x\in K$ is a defect group of $B$. 
For $x\in K$ let
\[\C_G(x)^*:=\{g\in G:gxg^{-1}=x^{\pm1}\}\le G\] 
be the \emph{extended centralizer} of $x$.

\begin{Prop}[Gow, Murray]\label{GowEx}
Every real $2$-block $B$ has a real defect class $K$. Let $x\in K$. Choose a Sylow $2$-subgroup $E$ of $\C_G(x)^*$ and put $D:=E\cap\C_G(x)$. Then the $G$-conjugacy class of the pair $(D,E)$ does not depend on the choice of $K$ or $x$.  
\end{Prop}
\begin{proof}
For the principal block (which is always real since it contains the trivial character), $K=\{1\}$ is a real defect class and $E=D$ is a Sylow $2$-subgroup of $G$. Hence, the uniqueness follows from Sylow's theorem. Now suppose that $B$ is non-principal.
The existence of $K$ was first shown in \cite[Theorem~5.5]{Gow}. Let $L$ be another real defect class of $B$ and choose $y\in L$. By \cite[Corollary~2.2]{GowExtended}, we may assume after conjugation that $E$ is also a Sylow $2$-subgroup of $\C_G(y)^*$. 
Let $D_x:=E\cap\C_G(x)$ and $D_y:=E\cap\C_G(y)$. We may assume that $|E:D_x|=2=|E:D_y|$ (cf. the remark after the proof).

We now introduce some notation in order to apply \cite[Proposition~14]{MurrayStronglyReal}. 
Let $\Sigma=\langle\sigma\rangle\cong C_2$. We consider $FG$ as an $F[G\times\Sigma]$-module where $G$ acts by conjugation and $g^\sigma=g^{-1}$ for $g\in G$ (observe that these actions indeed commute). For $H\le G\times\Sigma$ let \[\mathrm{Tr}^{G\times\Sigma}_H:(FG)^H\to(FG)^{G\times\Sigma},\ \alpha\mapsto\sum_{x\in \mathcal{R}}\alpha^x\] 
be the \emph{relative trace} with respect to $H$, where $\mathcal{R}$ denotes a set of representatives of the right cosets of $H$ in $G\times\Sigma$. 
By \cite[Proposition~14]{MurrayStronglyReal}, we have $1_B\in\mathrm{Tr}^{G\times\Sigma}_{E_x}(FG)$ where $E_x:=D_x\langle e_x\sigma\rangle$ for some $e_x\in E\setminus D_x$. By the same result we also obtain that $D_y\langle e_y\sigma\rangle$ with $e_y\in E\setminus D_y$ is $G$-conjugate to $E_x$. This implies that $D_y$ is conjugate to $D_x$ inside $\N_G(E)$. In particular, $(D_x,E)$ and $(D_y,E)$ are $G$-conjugate as desired.
\end{proof}

\begin{Def}\label{defdefectpair}
In the situation of \autoref{GowEx} we call $E$ an \emph{extended defect group} and $(D,E)$ a \emph{defect pair} of $B$.
\end{Def}

We stress that real $2$-blocks can have non-real defect classes and non-real blocks can have real defect classes (see \cite[Theorem~3.5]{GowMurray}). 

It is easy to show that non-principal real $2$-blocks cannot have maximal defect (see \cite[Problem~3.8]{Navarro}). In particular, the trivial class cannot be a defect class and consequently, $|E:D|=2$ in those cases. For non-real blocks we define the extended defect group by $E:=D$ for convenience. 
Every given pair of $2$-groups $D\le E$ with $|E:D|=2$ occurs as a defect pair of a real (nilpotent) block. To see this, let $Q\cong C_3$ and $G=Q\rtimes E$ with $\C_E(Q)=D$. Then $G$ has a unique non-principal block with defect pair $(D,E)$. 

We recall from \cite[p. 49]{Isaacs} that
\begin{equation}\label{indformula}
\sum_{\chi\in\Irr(G)}\epsilon(\chi)\chi(g)=|\{x\in G:x^2=g\}|
\end{equation}
for all $g\in G$.
The following proposition provides some interesting properties of defect pairs. 

\begin{Prop}[Gow, Murray]\label{GowMurray}
Let $B$ be a real $2$-block with defect pair $(D,E)$. 
Let $b_D$ be a Brauer correspondent of $B$ in $D\C_G(D)$. Then the following holds:
\begin{enumerate}[(i)]
\item $\N_G(D,b_D)^*=\N_G(D,b_D)E$. In particular, $b_D$ is real if and only if $E=D\C_E(D)$.

\item For $u\in D$, we have $\sum_{\chi\in\Irr(B)}\epsilon(\chi)\chi(u)\ge 0$ with strict inequality if and only if $u$ is $G$-conjugate to $e^2$ for some $e\in E\setminus D$. In particular, $E$ splits over $D$ if and only if $\sum_{\chi\in\Irr(B)}\epsilon(\chi)\chi(1)>0$.  

\item $E/D'$ splits over $D/D'$ if and only if all height zero characters in $\Irr(B)$ have non-negative F-S indicator.
\end{enumerate}
\end{Prop}
\begin{proof}\hfill
\begin{enumerate}[(i)]
\item See \cite[Lemma~1.8]{MurraySubpairs} and \cite[Theorem~1.4]{MurrayCyclic2}.

\item See \cite[Lemma~1.3]{MurraySubpairs}.

\item See \cite[Theorem~5.6]{Gow}.\qedhere
\end{enumerate}
\end{proof}

The next proposition extends \cite[Lemma~1.3]{MurrayCyclic2}. 

\begin{Cor}\label{cor}
Suppose that $B$ is a $2$-block with defect pair $(D,E)$ where $D$ is abelian. Then $E$ splits over $D$ if and only if all characters in $\Irr(B)$ have non-negative F-S indicator.
\end{Cor}
\begin{proof}
If $B$ is non-real, then $E=D$ splits over $D$ and all characters in $\Irr(B)$ have F-S indicator $0$. Hence, let $\overline{B}=B$. 
By Kessar--Malle~\cite{KessarMalle}, all characters in $\Irr(B)$ have height $0$. Hence, the claim follows from \autoref{GowMurray}(iii).
\end{proof}

\begin{Thm}\label{abel}
Let $B$ be a real, nilpotent $2$-block with defect pair $(D,E)$ where $D$ is abelian. If $E$ splits over $D$, then all real characters in $\Irr(B)$ have F-S indicator $1$. Otherwise exactly half of the real characters have F-S indicator $1$. In either case, Conjecture~B holds for $B$. 
\end{Thm}
\begin{proof}
If $E$ splits over $D$, then all real characters in $\Irr(B)$ have F-S indicator $1$ by \autoref{cor}. Otherwise we have $\sum_{\chi\in\Irr(B)}\epsilon(\chi)=0$ by \autoref{GowMurray}(ii), because all characters in $\Irr(B)$ have the same degree. Hence, exactly half of the real characters have F-S indicator $1$. Using Theorem~A we can determine the number of characters for each F-S indicator. 
For the last claim, we may therefore replace $B$ by the unique non-principal block of $G=Q\rtimes E$ where $Q\cong C_3$ and $\C_E(Q)=D$ (mentioned above). In this case Conjecture~B follows from Gow~\cite[Lemma~2.2]{Gow} or Theorem~E.
\end{proof}

\newlength{\myl}
\setlength{\myl}{3.6mm}
\begin{Ex}
Let $B$ be a real block with defect group $D\cong C_4\times C_2$. Then $B$ is nilpotent since $\Aut(D)$ is a $2$-group and $D$ is abelian. Moreover $|\Irr(B)|=8$. The F-S indicators depend not only on $E$, but also on the way $D$ embeds into $E$. The following cases can occur (here $M_{16}$ denotes the modular group and $[16,3]$ refers to the small group library):
\[\begin{array}{ll}
\text{F-S indicators}&E\\\hline
++++++++&D_8\times C_2\\
++++----&Q_8\times C_2,\ C_4\rtimes C_4\text{ with }\Phi(D)=E'\\
++++\parbox{\myl}{\centering$0$}\parbox{\myl}{\centering$0$}\parbox{\myl}{\centering$0$}\parbox{\myl}{\centering$0$}&D,\ D\times C_2,\ D_8*C_4,\ [16,3]\\
++--\parbox{\myl}{\centering$0$}\parbox{\myl}{\centering$0$}\parbox{\myl}{\centering$0$}\parbox{\myl}{\centering$0$}&C_4^2,\ C_8\times C_2,\ M_{16},\ C_4\rtimes C_4\text{ with }\Phi(D)\ne E'
\end{array}\]
\end{Ex}

The F-S indicator $\epsilon(\Phi)$ appearing in Conjecture~C has an interesting interpretation as follows. Let $\Omega:=\{g\in G:g^2=1\}$. The conjugation action of $G$ on $\Omega$ turns $F\Omega$ into an $FG$-module, called the \emph{involution module}.

\begin{Lem}[Murray]\label{Murray}
Let $B$ be a real $2$-block and $\phi\in\IBr(B)$. Then $\epsilon(\Phi_\phi)$ is the multiplicity of $\phi$ as a constituent of the Brauer character of $F\Omega$. 
\end{Lem}
\begin{proof}
See \cite[Lemma~2.6]{MurrayCyclic2}.
\end{proof}

Next we develop a local version of Conjecture~C. Let $B$ be a real $2$-block with defect pair $(D,E)$ and $B$-subsection $(u,b)$.
If $E=D\C_E(u)$, then $b$ is real and $(\C_D(u),\C_E(u))$ is a defect pair of $b$ by \cite[Lemma~2.6]{MurraySubpairs} applied to the subpair $(\langle u\rangle,b)$. Conversely, if $b$ is real, we may assume that $(\C_D(u),\C_E(u))$ is a defect pair of $b$ by \cite[Theorem~2.7]{MurraySubpairs}.
If $b$ is non-real, we may assume that $(\C_D(u),\C_D(u))=(\C_D(u),\C_E(u))$ is a defect pair of $b$.

\begin{Thm}\label{loc}
Let $B$ be $2$-block of a finite group $G$ with defect pair $(D,E)$. Suppose that Conjecture~C holds for all $2$-blocks of sections of $G$. 
Let $(u,b)$ be a $B$-subsection with defect pair $(\C_D(u),\C_E(u))$ such that $\IBr(b)=\{\phi\}$. Then
\[\sum_{\chi\in\Irr(B)}\epsilon(\chi)d^u_{\chi\phi}=\begin{cases}
|\{x\in D:x^2=u\}|&\text{if $B$ is the principal block},\\
|\{x\in E\setminus D:x^2=u\}|&\text{otherwise}.
\end{cases}\]
\end{Thm}
\begin{proof}
If $B$ is not real, then $B$ is non-principal and $E=D$. It follows that $\epsilon(\chi)=0$ for all $\chi\in\Irr(B)$ and 
\[|\{x\in E\setminus D:x^2=u\}|=0.\] 
Hence, we may assume that $B$ is real. 
By \autoref{lembrauer}, we have
\begin{equation}\label{brauer}
\sum_{\chi\in\Irr(B)}\epsilon(\chi)d^u_{\chi\phi}=\sum_{\psi\in\Irr(b)}\epsilon(\psi)d^u_{\psi\phi}=\frac{1}{\phi(1)}\sum_{\psi\in\Irr(b)}\epsilon(\psi)\psi(u).
\end{equation}
Suppose that $B$ is the principal block. Then $b$ is the principal block of $\C_G(u)$ by Brauer's third main theorem (see \cite[Theorem~6.7]{Navarro}). The hypothesis $l(b)=1$ implies that $\phi=1_{\C_G(u)}$ and $\C_G(u)$ has a normal $2$-complement $N$ (see \cite[Corollary~6.13]{Navarro}). It follows that $\Irr(b)=\Irr(\C_G(u)/N)=\Irr(\C_D(u))$ and 
\[\sum_{\psi\in\Irr(b)}\epsilon(\psi)d^u_{\psi\phi}=\sum_{\lambda\in\Irr(\C_D(u))}\epsilon(\lambda)\lambda(u)=|\{x\in\C_D(u):x^2=u\}|\]
by \eqref{indformula}. Since every $x\in D$ with $x^2=u$ lies in $\C_D(u)$, we are done in this case.

Now let $B$ be a non-principal real $2$-block. 
If $b$ is not real, then \eqref{brauer} shows that $\sum_{\chi\in\Irr(B)}\epsilon(\chi)d^u_{\chi\phi}=0$. 
On the other hand, we have $\C_E(u)=\C_D(u)\le D$ and $|\{x\in E\setminus D:x^2=u\}|=0$. Hence, we may assume that $b$ is real. 
Since every $x\in E$ with $x^2=u$ lies in $\C_E(u)$, we may assume that $u\in\Z(G)$ by \eqref{brauer}. 

Then $\chi(u)=d^u_{\chi\phi}\phi(1)$ for all $\chi\in\Irr(B)$. 
If $u^2\notin\Ker(\chi)$, then $\chi(u)\notin\RR$ and $\epsilon(\chi)=0$. Thus, it suffices to sum over $\chi$ with $d^u_{\chi\phi}=\pm d_{\chi\phi}$.
Let $Z:=\langle u\rangle\le\Z(G)$ and $\overline{G}:=G/Z$. Let $\hat{B}$ be the unique (real) block of $\overline{G}$ dominated by $B$. By \cite[Lemma~1.7]{MurraySubpairs}, $(\overline{D},\overline{E})$ is a defect pair for $\hat{B}$. 
Then, using \cite[Lemma~4.7]{Isaacs} and Conjecture~C for $B$ and $\hat{B}$, we obtain
\begin{align*}
\sum_{\chi\in\Irr(B)}\epsilon(\chi)d^u_{\chi\phi}&=\sum_{\chi\in\Irr(B)}\epsilon(\chi)(d_{\chi\phi}+d^u_{\chi\phi})-\sum_{\chi\in\Irr(B)}\epsilon(\chi)d_{\chi\phi}\\
&=2\sum_{\chi\in\Irr(\hat{B})}\epsilon(\chi)d_{\chi\phi}-\sum_{\chi\in\Irr(B)}\epsilon(\chi)d_{\chi\phi}\\
&=2|\{\overline{x}\in\overline{E}\setminus\overline{D}:\overline{x}^2=1\}|-|\{x\in E\setminus D:x^2=1\}|\\
&=\sum_{\lambda\in\Irr(E)}\epsilon(\lambda)(\lambda(1)+\lambda(u))-\sum_{\lambda\in\Irr(D)}\epsilon(\lambda)(\lambda(1)+\lambda(u))\\
&\quad -\sum_{\lambda\in\Irr(E)}\epsilon(\lambda)\lambda(1)+\sum_{\lambda\in\Irr(D)}\epsilon(\lambda)\lambda(1)\\
&=\sum_{\lambda\in\Irr(E)}\epsilon(\lambda)\lambda(u)-\sum_{\lambda\in\Irr(D)}\epsilon(\lambda)\lambda(u)=|\{x\in E\setminus D:x^2=u\}|.\qedhere
\end{align*}
\end{proof}

\section{Theorems~D and E}\label{secB}

The following result implies Theorem~D.

\begin{Thm}\label{thmBC}
Suppose that $B$ is a real, nilpotent, non-principal $2$-block fulfilling the statement of \autoref{loc}. Then Conjecture~B holds for $B$.
\end{Thm}
\begin{proof}
Let $(D,E)$ be defect pair of $B$. By Gow~\cite[Theorem~5.1]{Gow}, there exists a $2$-rational character $\chi_0\in\Irr(B)$ of height $0$ and $\epsilon(\chi_0)=1$. Let 
\[\Gamma:\Irr(D)\to\Irr(B),\qquad\lambda\mapsto\lambda*\chi_0\] 
be the Broué--Puig bijection.
Let $(u_1,b_1),\ldots,(u_k,b_k)$ be representatives for the conjugacy classes of $B$-subsections. Since $B$ is nilpotent, we may assume that $u_1,\ldots,u_k\in D$ represent the conjugacy classes of $D$. Let $\IBr(b_i)=\{\phi_i\}$ for $i=1,\ldots,k$.
Since $\chi_0$ is $2$-rational, we have $\sigma_i:=d^u_{\chi_0,\phi_i}\in\{\pm1\}$ for $i=1,\ldots,k$. 
Hence, the generalized decomposition matrix of $B$ has the form 
\[Q=(\lambda(u_i)\sigma_i:\lambda\in\Irr(D),i=1,\ldots,k)\]
(see \cite[Section~8.10]{LinckelmannBook2}). 
Let $v:=(\epsilon(\Gamma(\lambda)):\lambda\in\Irr(D))$ and $w:=(w_1,\ldots,w_k)$ where $w_i:=|\{x\in E\setminus D:x^2=u_i\}|$. Then \autoref{loc} reads as $vQ=w$. 

Let $d_i:=|\C_D(u_i)|$ and $d=(d_1,\ldots,d_k)$. Then the second orthogonality relation yields $Q^\TT\overline{Q}=\diag(d)$ where $Q^\TT$ denotes the transpose of $Q$. It follows that $Q^{-1}=\diag(d)^{-1}\overline{Q}^\TT$ and \[v=w\diag(d)^{-1}\overline{Q}^\TT=w\diag(d)^{-1}Q^\TT,\] 
because $\overline{v}=v$. Since $w_i=|\{x\in E\setminus D:x^2=u_i^y\}|$ for every $y\in D$, we obtain $\sum_{i=1}^kw_i|D:\C_D(u_i)|=|E\setminus D|=|D|$.
In particular,
\[1=\epsilon(\chi_0)=\sum_{i=1}^k\frac{w_i\sigma_i}{|\C_D(u_i)|}\le \sum_{i=1}^k\frac{w_i|\sigma_i|}{|\C_D(u_i)|}=1.\]
Therefore, $\sigma_i=1$ or $w_i=0$ for each $i$. This means that the signs $\sigma_i$ have no impact on the solution of the linear system $xQ=w$. Hence, we may assume that $Q=(\lambda(u_i))$ is just the character table of $D$. Since $Q$ has full rank, $v$ is the only solution of $xQ=w$. 
Setting $\mu(\lambda):=\frac{1}{|D|}\sum_{e\in E\setminus D}\lambda(e^2)$, it suffices to show that $(\mu(\lambda):\lambda\in\Irr(D))$ is another solution of $xQ=w$. Indeed,
\begin{align*}
\sum_{\lambda\in\Irr(D)}\frac{\lambda(u_i)}{|D|}\sum_{e\in E\setminus D}\lambda(e^2)&=\frac{1}{|D|}\sum_{e\in E\setminus D}\sum_{\lambda\in\Irr(D)}\lambda(u_i)\lambda(e^2)\\
&=\frac{1}{|D|}\sum_{\substack{e\in E\setminus D\\e^2=u_i^{-1}}}|D:\C_D(u_i)||\C_D(u_i)|=w_i
\end{align*}
for $i=1,\ldots,k$. 
\end{proof}

\begin{ThmE}
Conjectures B and C hold for all nilpotent $2$-blocks of solvable groups. 
\end{ThmE}
\begin{proof}
Let $B$ be a real, nilpotent, non-principal $2$-block of a solvable group $G$ with defect pair $(D,E)$. We first prove Conjecture~C for $B$. Since all sections of $G$ are solvable and all blocks dominated by $B$-subsections are nilpotent, Conjecture~C holds for those blocks as well. Hence, the hypothesis of \autoref{loc} is fulfilled for $B$. Now by \autoref{thmBC}, Conjecture~B holds for $B$. 

Let $N:=\pcore_{2'}(G)$ and let $\theta\in\Irr(N)$ such that the block $\{\theta\}$ is covered by $B$. 
Since $B$ is non-principal, $\theta\ne 1_N$ and therefore $\overline{\theta}\ne\theta$ as $N$ has odd order.
Since $B$ also lies over $\overline{\theta}$, it follow that $G_\theta<G$. Let $b$ be the Fong--Reynolds correspondent of $B$ in the extended stabilizer $G_\theta^*$.
By \cite[Theorem~9.14]{Navarro} and \cite[p. 94]{MurrayCyclicOdd}, the Clifford correspondence $\Irr(b)\to\Irr(B)$, $\psi\mapsto\psi^G$ preserves decomposition numbers and F-S indicators. Thus, we need to show that $b$ has defect pair $(D,E)$. Let $\beta$ be the Fong--Reynolds correspondent of $B$ in $G_\theta$. By \cite[Theorem~10.20]{Navarro}, $\beta$ is the unique block over $\theta$. In particular, the block idempotents $1_\beta=1_\theta$ are the same (we identify $\theta$ with the block $\{\theta\}$). Since $b$ is also the unique block of $G_\theta^*$ over $\theta$, we have $1_b=1_\theta+1_{\overline{\theta}}=\sum_{x\in N}\alpha_xx$ for some $\alpha_x\in F$. Let $S$ be a set of representatives for the cosets $G/G_\theta^*$. Then
\[1_B=\sum_{s\in S}(1_\theta+1_{\overline{\theta}})^s=\sum_{s\in S}1_b^s=\sum_{g\in N}\Bigl(\sum_{s\in S}\alpha_{g^{s^{-1}}}\Bigr)g.\]
Hence, there exists a real defect class $K$ of $B$ such that $\alpha_{g^{s^{-1}}}\ne 0$ for some $g\in K$ and $s\in S$. Of course we can assume that $g=g^{s^{-1}}$. Then $1_b$ does not vanish on $g$. By \cite[Theorem~9.1]{Navarro}, the central characters $\lambda_B$, $\lambda_b$ and $\lambda_\theta$ agree on $N$. It follows that $K$ is also a real defect class of $b$. Hence, we may assume that $(D,E)$ is a defect pair of $b$. 

It remains to consider $G=G_\theta^*$ and $B=b$. Then $D$ is a Sylow $2$-subgroup of $G_\theta$ by \cite[Theorem~10.20]{Navarro} and $E$ is a Sylow $2$-subgroup of $G$. 
Since $|G:G_\theta|=2$, it follows that $G_\theta\unlhd G$ and $N=\pcore_{2'}(G_\theta)$. By \cite[Lemma~1 and 2]{NavarroNil}, $\beta$ is nilpotent and $G_\theta$ is $2$-nilpotent, i.\,e. $G_\theta=N\rtimes D$ and $G=N\rtimes E$. 
Let $\widetilde{\Phi}:=\sum_{\chi\in\Irr(B)}\chi(1)\chi=\phi(1)\Phi$ where $\IBr(B)=\{\phi\}$. 
We need to show that
\[\epsilon(\widetilde\Phi)=\phi(1)|\{x\in E\setminus D:x^2=1\}|.\]
Note that $\chi_N=\frac{\chi(1)}{2\theta(1)}(\theta+\overline{\theta})$. 
By Frobenius reciprocity, it follows that $\widetilde\Phi=2\theta(1)\theta^G$ and 
\[\widetilde\Phi_N=|G:N|\theta(1)(\theta+\overline{\theta}).\] 
Since $\Phi$ vanishes on elements of even order, $\widetilde\Phi$ vanishes outside $N$. Since $\widetilde{\Phi}_{G_\theta}$ is a sum of non-real characters in $\beta$, we have
\[\epsilon(\widetilde\Phi)=\frac{1}{|G|}\sum_{g\in G_\theta}\widetilde\Phi(g^2)+\frac{1}{|G|}\sum_{g\in G\setminus G_\theta}\widetilde\Phi(g^2)=\frac{1}{|G|}\sum_{g\in G\setminus G_\theta}\widetilde\Phi(g^2).\]
Every $g\in G\setminus G_\theta=NE\setminus ND$ with $g^2\in N$ is $N$-conjugate to a unique element of the form $xy$ where $x\in E\setminus D$ is an involution and $y\in\C_N(x)$ (Sylow's theorem). Setting $\Delta:=\{x\in E\setminus D:x^2=1\}$, we obtain
\begin{equation}\label{tildePhi}
\epsilon(\widetilde\Phi)=\frac{\theta(1)}{|N|}\sum_{x\in\Delta}|N:\C_N(x)|\sum_{y\in\C_N(x)}(\theta(y)+\overline{\theta(y)})=2\theta(1)\sum_{x\in\Delta}\frac{1}{|\C_N(x)|}\sum_{y\in\C_N(x)}\theta(y).
\end{equation}
For $x\in\Delta$ let $H_x:=N\langle x\rangle$. Again by Sylow's theorem, the $N$-orbit of $x$ is the set of involutions in $H_x$. 
From $\theta^x=\overline{\theta}$ we see that $\theta^{H_x}$ is an irreducible character of $2$-defect $0$. By \cite[Theorem~5.1]{Gow}, we have $\epsilon(\theta^{H_x})=1$. Now applying the same argument as before, it follows that
\[1=\epsilon(\theta^{H_x})=\frac{1}{|N|}\sum_{g\in H_x\setminus N}\theta^{H_x}(g^2)=\frac{2}{|\C_N(x)|}\sum_{y\in\C_N(x)}\theta(y).\]
Combined with \eqref{tildePhi}, this yields $\epsilon(\widetilde\Phi)=2\theta(1)|\Delta|$.
By Green's theorem (see \cite[Theorem~8.11]{Navarro}), $\phi_N=\theta+\overline{\theta}$ and $\epsilon(\widetilde{\Phi})=\phi(1)|\Delta|$ as desired.
\end{proof}

For non-principal blocks $B$ of solvable groups with $l(B)=1$ it is not true in general that $G_\theta$ is $2$-nilpotent in the situation of Theorem~E. For example, a (non-real) $2$-block of a triple cover of $A_4\times A_4$ has a unique simple module. Extending this group by an automorphism of order $2$, we obtain the group $G=\mathtt{SmallGroup}(864,3988)$, which fulfills the assumptions with $D\cong C_2^4$, $N\cong C_3$ and $|G:NE|=9$. 

In order to prove Conjecture~C for arbitrary $2$-blocks of solvable groups, we may follow the steps in the proof above and invoke a result on fully ramified Brauer characters \cite[Theorem~2.1]{NavarroFully}.
The claim then boils down to a purely group-theoretical statement:
Let $B$ be a real, non-principal $2$-block of a solvable group $G$ with defect pair $(D,E)$ and $l(B)=1$. Let $\overline{G}:=G/\pcore_{2'}(G)$. Then
\[|\{\overline{x}\in \overline{G}\setminus\overline{G_\theta}:\overline{x}^2=1\}|=|\{x\in E\setminus D:x^2=1\}|\sqrt{|G:EN|}.\]
Unfortunately, I am unable to prove this.

\section*{Acknowledgment}
I thank Gabriel Navarro for providing some arguments for Theorem~E from his paper~\cite{NavarroNil}. John Murray and three anonymous referees have made many valuable comments, which improved the quality of the manuscript. 
The work is supported by the German Research Foundation (\mbox{SA 2864/3-1} and \mbox{SA 2864/4-1}).

\end{document}